\documentclass[jip]{degruyter-journal-a}      

\newcommand{\set}[1]{\left\{#1\right\}}
\newcommand{\abs}[1]{\left|#1\right|}
\newcommand{\p}{\partial}

\newcommand{\mU}{\mathbf{U}}

\newcommand{\mf}{\mathbf{f}}

\newcommand{\mx}{\mathbf{x}}
\newcommand{\my}{\mathbf{y}}
\newcommand{\mz}{\mathbf{z}}

\newcommand{\vt}{\boldsymbol{\theta}}
\newcommand{\vv}{\boldsymbol{\vartheta}}

\newtheorem{Theorem}{Theorem}[section]
\newtheorem{Lemma}[Theorem]{Lemma}
\newtheorem{Remark}[Theorem]{Remark}

\title{Appearance of inaccurate results in the MUSIC algorithm with inappropriate wavenumber}
\headlinetitle{MUSIC algorithm with inappropriate wavenumber}

\lastnameone{Park}
\firstnameone{Won-Kwang}
\nameshortone{W.-K. Park}
\addressone{77, Jeongneung-ro, Seongbuk-gu, Seoul, 02707}
\countryone{Republic of Korea}
\emailone{parkwk@kookmin.ac.kr}

\abstract{MUltiple SIgnal Classification (MUSIC) is a well-known non-iterative location detection algorithm for small, perfectly conducting cracks in inverse scattering problems. However, when the applied wavenumbers are unknown, inaccurate locations of targets are extracted by MUSIC with inappropriate wavenumbers, a fact that has been confirmed by numerical simulations.  To date, the reason behind this phenomenon has not been theoretically investigated. Motivated by this fact, we identify the structure of MUSIC-type imaging functionals with inappropriate wavenumbers by establishing a relationship with Bessel functions of order zero of the first kind. This result explains the reasons for inaccurate results. Various results of numerical simulations with noisy data support the identified structure of MUSIC.}



\researchsupported{This research was supported by Basic Science Research Program through the National Research Foundation of Korea (NRF) funded by the Ministry of Education (No. NRF-2014R1A1A2055225).}

\begin{document}

\section{INTRODUCTION}
Inverse scattering problem for imaging a single, perfectly conducting crack in $\mathbb{R}^2$ satisfying a Dirichlet boundary condition has been studied in \cite{K}. In this remarkable research, Newton-type iterative reconstruction algorithm has been suggested. Generally, for a successful application, a good initial guess close to the unknown crack is needed to guarantee convergence. Furthermore. it generally requires a large amount of computational time, and it is hard to extend to the reconstruction of multiple cracks.

For an alternative, non-iterative algorithms have been developed. Among them, MUltiple SIgnal Classification (MUSIC)-type algorithm is applied to the problem for detecting small inhomogeneities or crack-like defects, refer to \cite{AGKPS,AKLP,C1,CZ,K1,P-MUSIC1,PL1,PL3,ZC} and references therein. However, for obtaining a good result, the value of applied wavenumber must be known. If not, an inaccurate locations or shapes of targets are extracted via MUSIC. Until now, this fact has been examined through the results of numerical simulations. In recent work \cite{SRACM}, an analysis of MUSIC for detecting point-like scatterers has been considered but a suitable mathematical theory about the structure of MUSIC-type imaging functional for detecting cracks must be established to diagnose such phenomenon.

In this contribution, we identify the structure of MUSIC-type imaging functional of small, perfectly conducting cracks with unknown wavenumber by finding a relationship with Bessel function of order zero of the first kind. This is based on the fact that the far-field pattern can be presented as an asymptotic expansion formula in the existence of small crack. Derived structure tells us theoretical reason of appearance of inaccurate locations through MUSIC algorithm.

The rest of the investigation is arranged as follows. In Section \ref{sec2}, two-dimensional direct scattering problem and MUSIC-type imaging algorithm are introduced briefly. The structure of MUSIC-type imaging functional without information of applied wavenumber is identified in Section \ref{sec3}. In Section \ref{sec4}, corresponding results of numerical simulation is exhibited. Finally, a short conclusion is mentioned in Section \ref{sec5}.

\section{DIRECT SCATTERING PROBLEM AND MUSIC ALGORITHM}\label{sec2}
Let $\Gamma_m$, where $m=1,2,\cdots,M$, be a linear crack of length $2h$ centered at $\mz_m$, and let $\Gamma$ be the collection of $\Gamma_m$. Assume that $\Gamma_m$ are sufficiently separated from each other.

In this paper, we consider Transverse Magnetic (TM) polarization. Let $u(\mx,\vt)\in\mathcal{C}^2(\mathbb{R}^2\backslash\Gamma)\cap\mathcal{C}(\mathbb{R}^2)$ be the time-harmonic total field that satisfies the following Helmholtz equation:
\begin{align}
\begin{aligned}\label{Helmholtz}
  \left\{\begin{array}{rcl}
  \triangle u(\mx,\vt)+k^2u(\mx,\vt)=0&\mbox{in}&\mathbb{R}^2\backslash\Gamma,\\
  u(\mx,\vt)=0&\mbox{on}&\Gamma.
  \end{array}\right.,
\end{aligned}
\end{align}
where $\vt$ is an incident direction on the two-dimensional unit circle $\mathbb{S}^1$ centered at the origin, and $k=2\pi/\lambda$ denotes a strictly positive wavenumber with wavelength $\lambda$. Throughout this paper, we assume that $k^2$ is not an eigenvalue of (\ref{Helmholtz}), $h\ll\lambda$, and sufficiently large such that
\begin{equation}\label{separated}
  k|\mz_m-\mz_{m'}|\gg0.25.
\end{equation}

Note that $u(\mx,\vt)$ can be decomposed as $u(\mx,\vt)=u_{\mathrm{inc}}(\mx,\vt)+u_{\mathrm{scat}}(\mx,\vt)$, where $u_{\mathrm{inc}}(\mx,\vt)=e^{ik\vt\cdot\mx}$ is the given incident field, and $u_{\mathrm{scat}}(\mx,\vt)\in\mathcal{C}^2(\mathbb{R}^2\backslash\Gamma)\cap\mathcal{C}(\mathbb{R}^2)$ is the unknown scattered field that satisfies the Sommerfeld radiation condition
\[\lim_{\abs{\mx}\to\infty}\sqrt{\abs{\mx}}\left(\frac{\p u_{\mathrm{scat}}(\mx,\vt)}{\p\abs{\mx}}-iku_{\mathrm{scat}}(\mx,\vt)\right)=0\]
uniformly in all directions $\vv=\mx/\abs{\mx}$.

The far-field pattern $u_{\infty}(\vv,\vt)$ of the scattered field $u_{\mathrm{scat}}(\mx,\vt)$ is defined on the two-dimensional unit circle $\mathbb{S}^1$. It can be represented as
\[u_{\mathrm{scat}}(\mx,\vt)=\frac{e^{ik\abs{\mx}}}{\sqrt{\abs{\mx}}}\left\{u_{\infty}(\vv,\vt)+O\left(\frac{1}{\abs{\mx}}\right)\right\}\]
uniformly in all directions $\vv=\mx/\abs{\mx}$ and $\abs{\mx}\longrightarrow\infty$.

From \cite{AKLP}, note that $u_{\infty}(\vv,\vt)$ can be represented by the following asymptotic form.

\begin{Lemma}
  Let $u(\mx,\vt)$ satisfy (\ref{Helmholtz}). Then the following asymptotic expansion formula holds uniformly for $0<h<2$ and $h\ll\lambda$:
  \begin{align}
  \begin{aligned}\label{AsymptoticExpansionFormula}
    u_\infty(\vv,\vt)&=-\frac{2\pi}{\ln(h/2)}\sum_{m=1}^{M}u_{\mathrm{inc}}(\mz_m,\vt)\overline{u_{\mathrm{inc}}(\mz_m,\vv)}+O\bigg(\frac{1}{|\ln h|^2}\bigg)\\
    &\approx-\frac{2\pi}{\ln(h/2)}\sum_{m=1}^{M}e^{ik(\vt-\vv)\cdot\mz_m}.
  \end{aligned}
  \end{align}
\end{Lemma}

Equation (\ref{AsymptoticExpansionFormula}) can be applied to establish a MUSIC-type imaging functional. To this end, the eigenvalue structure of the modified sparse row (MSR) matrix $\mathbb{K}=\left[u_\infty(\vv_j,\vt_l)\right]_{j,l=1}^{N}$ is utilized. Suppose $\vv_j=-\vt_j$ for all $j$. Then $\mathbb{K}$ is a complex symmetric matrix but not Hermitian; thus, instead of eigenvalue decomposition, we perform the Singular Value Decomposition (SVD) of $\mathbb{K}$ (see \cite{C}):
\begin{equation}\label{SVD}
  \mathbb{K}=\sum_{m=1}^{M}\sigma_m\mathbf{U}_m\mathbf{V}_m^*,
\end{equation}
where the superscript $*$ denotes the Hermitian. Then $\set{\mathbf{U}_1,\mathbf{U}_2,\cdots,\mathbf{U}_M}$ is the basis for the signal space of $\mathbb{K}$. Therefore, one can define the projection operator onto the null (or noise) subspace, $\mathbf{P}_{\mathrm{noise}}:\mathbb{C}^{N\times1}\longrightarrow\mathbb{C}^{N\times1}$. This projection is given explicitly by
\begin{equation}\label{Projection}
\mathbf{P}_{\mathrm{noise}}:=\mathbb{I}(N)-\sum_{m=1}^{M}\mathbf{U}_m\mathbf{U}_m^*,
\end{equation}
where $\mathbb{I}(N)$ denotes the $N\times N$ identity matrix. For any point $\mx\in\mathbb{R}^2$, a test vector $\mf(\mx;k)\in\mathbb{C}^{N\times1}$ can be defined as
\begin{equation}\label{testvector}
  \mf(\mx;k)=\frac{1}{N}\bigg[e^{i k\vt_1\cdot\mx},e^{i k\vt_2\cdot\mx},\cdots,e^{i k\vt_N\cdot\mx}\bigg]^T.
\end{equation}
As a result, the MUSIC-type imaging functional $\mathbb{E}(\mx;k)$ can be designed as follows:
\begin{equation}\label{MUSICimaging}
  \mathbb{E}(\mx;k)=\frac{1}{|\mathbf{P}_{\mathrm{noise}}(\mf(\mx;k))|}.
\end{equation}
The map of $\mathbb{E}(\mx;k)$ will have peaks of large and small magnitudes at $\mz_m\in\Gamma$ and $\mz_m\in\mathbb{R}^2\backslash\Gamma$, respectively.

\section{STRUCTURE OF IMAGING FUNCTIONAL}\label{sec3}
In light of the discussion in the previous section, one must know the exact value of $k$; if the exact value is not known, the exact locations of $\Gamma_m$ cannot be detected. This fact has already been identified in previous research through numerical simulations. Therefore, we explore the structure of (\ref{MUSICimaging}) with an unknown wavenumber. First, recall the following results.

\begin{Lemma}[See \cite{AGKPS}]
  For $\vt_n\in\mathbb{S}^1$, $n=1,2,\cdots,N$, the left singular vectors $\mU_m$ of the MSR matrix $\mathbb{K}$ is of the form
  \begin{equation}\label{SingularVectorRelation}
    \mU_m\approx\mf(\mx;\omega)+\mathcal{O}\left(\frac{1}{|\ln h|^2}\right),
  \end{equation}
  for $m=1,2,\cdots,M$.
\end{Lemma}

\begin{Lemma}[See \cite{P-SUB3}]
  Assume $\set{\vt_n:n=1,2,\cdots,N}$ spans $\mathbb{S}^1$ such that
  \[\vt_n=\bigg[\cos\theta_n,\sin\theta_n\bigg]^T,\quad\theta_n=\theta_1+(\theta_N-\theta_1)\frac{n-1}{N-1},\]
  where $0=\theta_1<\theta_2<\cdots<\theta_{N-1}<\theta_N=2\pi$. Then for $\vt\in\mathbb{S}^1$, $\mx\in\mathbb{R}^2$, and sufficiently large $N$,
  \begin{equation}\label{RelationBessel}
    \frac{1}{N}\sum_{n=1}^{N}e^{i\omega\vt_n\cdot\mx}=\frac{1}{2\pi}\int_{\mathbb{S}^1}e^{i\omega\vt\cdot\mx}d\vt=J_0(\omega|\mx|),
  \end{equation}
  where $J_0$ denotes the Bessel function of order zero of the first kind.
\end{Lemma}

Since we assumed that the exact value of $k$ is unknown, let us choose a fixed value of $\eta$ instead of (\ref{testvector}) and apply the corresponding test vector \[\mf(\mx;\eta)=\frac{1}{N}\bigg[e^{i\eta\vt_1\cdot\mx},e^{i\eta\vt_2\cdot\mx},\cdots,e^{i\eta\vt_N\cdot\mx}\bigg]^T\]
to (\ref{MUSICimaging}) such that
\begin{equation}\label{MUSICimagingUnknown}
  \mathbb{E}(\mx;\eta)=\frac{1}{|\mathbf{P}_{\mathrm{noise}}(\mf(\mx;\eta))|}.
\end{equation}
Then, the following result can be obtained.

\begin{Theorem}\label{StructureMUSICTheorem}
  For sufficiently large $N$ and $\omega$, $\mathbb{E}(\mx;\eta)$ is of the form
  \[\mathbb{E}(\mx;\eta)\approx\left(1-\sum_{m=1}^{M}J_0(|\eta\mx-k\mz_m|)+\mathcal{O}\left(\frac{1}{|\ln h|^2}\right)\right)^{-1/2}.\]
\end{Theorem}
\begin{proof}
  Based on (\ref{SingularVectorRelation}), since $\mf(\mz_mk)\approx\mU_m$ for all $m=1,2,\cdots,M$, it follows that
  \begin{align*} \mathbf{P}_{\mathrm{noise}}&(\mf(\mx;\eta))=\left(\mathbb{I}(N)-\sum_{m=1}^{M}\mathbf{U}_m\mathbf{U}_m^*\right)\mf(\mx;\eta)\\
  \approx&\left(\mathbb{I}(N)-\sum_{m=1}^{M}\bigg(\mf(\mz_m;\omega)+\mathcal{O}(|\ln h|^{-2})\bigg)\bigg(\mf(\mz_m;\omega)+\mathcal{O}(|\ln h|^{-2})\bigg)^*\right)\mf(\mx;\eta)\\
    =&\frac{1}{\sqrt{N}}\left[
        \begin{array}{c}
          e^{i\eta\vt_1\cdot\mx} \\
          e^{i\eta\vt_2\cdot\mx} \\
          \vdots \\
          e^{i\eta\vt_N\cdot\mx}
        \end{array}
      \right]\\
      &-
      \frac{1}{N\sqrt{N}}\sum_{m=1}^{M}\left[
        \begin{array}{c}
          \medskip\displaystyle e^{i\eta\vt_1\cdot\mx}+\sum_{n\in\mathbb{N}_1}e^{i k\vt_1\cdot\mz_m}e^{i\vt_n\cdot(\eta\mx-k\mz_m)}+\mathcal{O}\left(\frac{1}{|\ln h|^2}\right) \\
          \medskip\displaystyle e^{i\eta\vt_2\cdot\mx}+\sum_{n\in\mathbb{N}_2}e^{i k\vt_2\cdot\mz_m}e^{i\vt_n\cdot(\eta\mx-k\mz_m)}+\mathcal{O}\left(\frac{1}{|\ln h|^2}\right) \\
          \medskip\vdots \\
          \medskip\displaystyle e^{i\eta\vt_N\cdot\mx}+\sum_{n\in\mathbb{N}_N}e^{i k\vt_N\cdot\mz_m}e^{i\vt_n\cdot(\eta\mx-k\mz_m)}+\mathcal{O}\left(\frac{1}{|\ln h|^2}\right)
        \end{array}
      \right],
  \end{align*}
  where $N_j:=\set{1,2,\cdots,N}\backslash\set{j}$ for $j=1,2,\cdots,N$. Since,
  \[e^{i\eta\vt_j\cdot\mx}=e^{ik\vt_j\cdot\mz_m}e^{i\vt_j\cdot(\eta\mx-k\mz_m)},\]
  applying (\ref{RelationBessel}), we can evaluate
  \begin{align*}
  e^{i\eta\vt_j\cdot\mx}+\sum_{n\in\mathbb{N}_j}e^{i k\vt_2\cdot\mz_m}e^{i\vt_n\cdot(\eta\mx-k\mz_m)}&=e^{ik\vt_j\cdot\mz_m}\sum_{n=1}^{N}e^{i\vt_n\cdot(\eta\mx-k\mz_m)}\\
  &=Ne^{ik\vt_j\cdot\mz_m}J_0(|\eta\mx-k\mz_m|).
  \end{align*}
  Hence,
      \[\mathbf{P}_{\mathrm{noise}}(\mf(\mx;\eta))=\frac{1}{\sqrt{N}}\left[
        \begin{array}{c}
          \medskip\displaystyle e^{i\eta\vt_1\cdot\mx}-\sum_{m=1}^{M}e^{i k\vt_1\cdot\mz_m}J_0(|\eta\mx-k\mz_m|)+\mathcal{O}\left(\frac{1}{|\ln h|^2}\right) \\
          \medskip\displaystyle e^{i\eta\vt_2\cdot\mx}-\sum_{m=1}^{M}e^{i k\vt_2\cdot\mz_m}J_0(|\eta\mx-k\mz_m|)+\mathcal{O}\left(\frac{1}{|\ln h|^2}\right) \\
          \medskip\vdots \\
          \medskip\displaystyle e^{i\eta\vt_N\cdot\mx}-\sum_{m=1}^{M}e^{i k\vt_N\cdot\mz_m}J_0(|\eta\mx-k\mz_m|)+\mathcal{O}\left(\frac{1}{|\ln h|^2}\right)
        \end{array}
      \right].\]
  Based on this result, it follows that
  \begin{align*}
    |\mathbf{P}_{\mathrm{noise}}(\mf(\mx;\eta))|&=\bigg(\mathbf{P}_{\mathrm{noise}}(\mf(\mx;\eta))\cdot\overline{\mathbf{P}_{\mathrm{noise}}(\mf(\mx;\eta))}\bigg)^{1/2}\\
    &=\left(\frac{1}{N}\sum_{n=1}^{N}\left\{1-(\Phi_1+\overline{\Phi}_1)+(\Phi_2\overline{\Phi}_2)+\mathcal{O}\left(\frac{1}{|\ln h|^2}\right)\right\}\right)^{1/2},
    \end{align*}
  where
  \begin{align*}
  \Phi_1&=\sum_{m=1}^{M}e^{i\vt_n\cdot(\eta\mx-k\mz_m)}J_0(|\eta\mx-k\mz_m|)\\
  \Phi_2&=\sum_{m=1}^{M}e^{i k\vt_n\cdot\mz_m}J_0(|\eta\mx-k\mz_m|).
  \end{align*}

  Applying (\ref{RelationBessel}) again, we see that
  \begin{multline*}
    \frac{1}{N}\sum_{n=1}^{N}\Phi_1=\frac{1}{N}\sum_{n=1}^{N}\sum_{m=1}^{M}e^{i\vt_n\cdot(\eta\mx-k\mz_m)}J_0(|\eta\mx-k\mz_m|)\\
    =\sum_{m=1}^{M}\left(\frac{1}{N}\sum_{n=1}^{N}e^{i\vt_n\cdot(\eta\mx-k\mz_m)}\right)J_0(|\eta\mx-k\mz_m|)=\sum_{m=1}^{M}J_0(|\eta\mx-k\mz_m|)^2.
  \end{multline*}
  Hence,
  \begin{equation}\label{term1}
    \frac{1}{N}\sum_{n=1}^{N}(\Phi_1+\overline{\Phi}_1)=2\sum_{m=1}^{M}J_0(|\eta\mx-k\mz_m|)^2.
  \end{equation}

  Since the singular vectors are orthonormal to each other, $\mU_m\cdot\mU_{m'}=0$ for $m,m'=1,2,\cdots,M$ and $m\ne m'$. Hence, based on (\ref{separated}),
  \begin{align*}
  \mU_m\cdot\mU_{m'}\approx\frac{1}{N}\sum_{n=1}^{N}e^{ik\vt_n\cdot(\mz_m-\mz_{m'})}=J_0(k|\mz_m-\mz_{m'}|)\\
  \approx\sqrt{\frac{2}{k\pi|\mz_m-\mz_{m'}|}}\cos\left(k|\mz_m-\mz_{m'}|-\frac{\pi}{4}\right)\longrightarrow0.
  \end{align*}
  For $m,m'=1,2,\cdots,M$ and $m\ne m'$,
  \begin{align}
  \begin{aligned}\label{term2}
    &\frac{1}{N}\sum_{n=1}^{N}\Phi_2\overline{\Phi}_2\\
    &=\frac{1}{N}\sum_{n=1}^{N}\left(\sum_{m=1}^{M}e^{ik\vt_n\cdot\mz_m}J_0(|\eta\mx-k\mz_m|)\right)\left(\sum_{m'=1}^{M}e^{-ik\vt_n\cdot\mz_{m'}}J_0(|\eta\mx-k\mz_{m'}|)\right)\\
    &=\frac{1}{N}\sum_{n=1}^{N}\sum_{m=1}^{M}\sum_{m'=1}^{M}e^{ik\vt_n\cdot\mz_m}J_0(|\eta\mx-k\mz_m|)e^{-ik\vt_n\cdot\mz_{m'}}J_0(|\eta\mx-k\mz_{m'}|)\\
    &=\sum_{m=1}^{M}\sum_{m'=1}^{M}\left(\frac{1}{N}\sum_{n=1}^{N}e^{ik\vt_n\cdot(\mz_m-\mz_{m'})}\right)J_0(|\eta\mx-k\mz_m|)J_0(|\eta\mx-k\mz_{m'}|)\\
    &=\sum_{m=1}^{M}\sum_{m'=1}^{M}J_0(k|\mz_m-\mz_{m'}|)J_0(|\eta\mx-k\mz_m|)J_0(|\eta\mx-k\mz_{m'}|)\\
    &=\sum_{m=1}^{M}J_0(|\eta\mx-k\mz_m|)^2.
  \end{aligned}
  \end{align}

  Therefore, by combining (\ref{term1}) and (\ref{term2}), $|\mathbf{P}_{\mathrm{noise}}(\mf(\mx;\eta))|$ can be represented as
  \[|\mathbf{P}_{\mathrm{noise}}(\mf(\mx;\eta))|=\left(1-\sum_{m=1}^{M}J_0(|\eta\mx-k\mz_m|)+\mathcal{O}\left(\frac{1}{|\ln h|^2}\right)\right)^{1/2}.\]
  This completes the proof.
\end{proof}

\begin{Remark}\label{Remark}
  On the basis of the result in Theorem \ref{StructureMUSICTheorem}, we can verify some some properties. They can be summarized as follows.
  \begin{enumerate}
    \item $J_0(x)$ has a maximum value of $1$ at $x=0$; therefore, the map of $\mathbb{E}(\mx;\eta)$ will have plots of the magnitude (theoretically $+\infty$) at $\mx=(k/\eta)\mz_m$ instead of at the true location $\mz_m$ if $h$ is small enough. This is the theoretical reason why inaccurate crack locations are extracted by MUSIC.
    \item If a crack (centered at $\mz$) is located at the origin, its true location can be identified exactly for any value of $\eta$ since $k\mz\equiv0$. Similarly, if a crack is located on the $x-$axis (or $y-$axis), one can find exact $y-$th (or $x-$th) component of location of crack.
    \item For any value of $\eta$, we can extract the information of existence and total number of cracks however, exact location of cracks cannot be extracted without \textit{a priori} information of true value of $k$. Therefore, an effective algorithm for estimating $k$ must be investigated for finding exact location. Corresponding algorithm will be discussed in the next section.
  \end{enumerate}
\end{Remark}

\section{RESULTS OF NUMERICAL SIMULATIONS AND FINDING EXACT LOCATIONS}\label{sec4}
In this section, we discuss the results of numerical simulations that support Theorem \ref{StructureMUSICTheorem}. In particular, $N=16$ different incident and observation directions were applied.

Based on \cite{N}, every far-field pattern $u_{\infty}(\vv_j,\vt_l)$, $j,l=1,2,\cdots,N$, was computed by a second-kind Fredholm integral equation along the crack to avoid \textit{inverse crime}. Next, $20$ dB white Gaussian random noise was added to the unperturbed data using the MATLAB \texttt{awgn} command.

Three cracks of length $h=0.05$ were chosen for numerical simulations:
\begin{align*}
  \Gamma_1&=\set{[s-0.6,-0.2]^T:-h\leq s\leq h}\\
  \Gamma_2&=\set{\mathrm{R}_{\pi/4}[s+0.4,s+0.35]^T:-h\leq s\leq h}\\
  \Gamma_3&=\set{\mathrm{R}_{7\pi/6}[s+0.25,s-0.6]^T:-h\leq s\leq h}.
\end{align*}
Here, $\mathrm{R}_{\varphi}$ denotes the rotation by $\varphi$.

Figure \ref{Result1} displays the maps of $\mathbb{E}(\mx;\eta)$ for three small cracks for the true value $k=2\pi/0.5\approx12.5664$. Following the result in Theorem \ref{StructureMUSICTheorem}, the locations of $(k/\eta)\mz_m$ were identified instead of $\mz_m$. Notice that the identified locations of $\Gamma_m$ are scattered when $\eta<k$ (see Figure \ref{Result1-1}) and concentrated when $\eta>k$ (see Figures \ref{Result1-2} and \ref{Result1-3}). Figure \ref{Result2} displays the maps of $\mathbb{E}(\mx;\eta)$ for three small cracks for the true value $k=2\pi/0.3\approx20.9440$. Similar to the results in Figure \ref{Result1}, the locations of $(k/\eta)\mz_m$ were identified instead of $\mz_m$. Notice that when $\eta=20$, since $\eta\approx k$, very accurate locations of $\Gamma_m$ are extracted. Based on the results in Figure \ref{Result1} and Figure \ref{Result2}, the existence and total number of cracks can be determined, but it is impossible to find true locations without \textit{a priori} information of $k$.

\begin{figure}
\begin{center}
\subfloat[$\eta=10$]{\label{Result1-1}\includegraphics[width=0.495\textwidth]{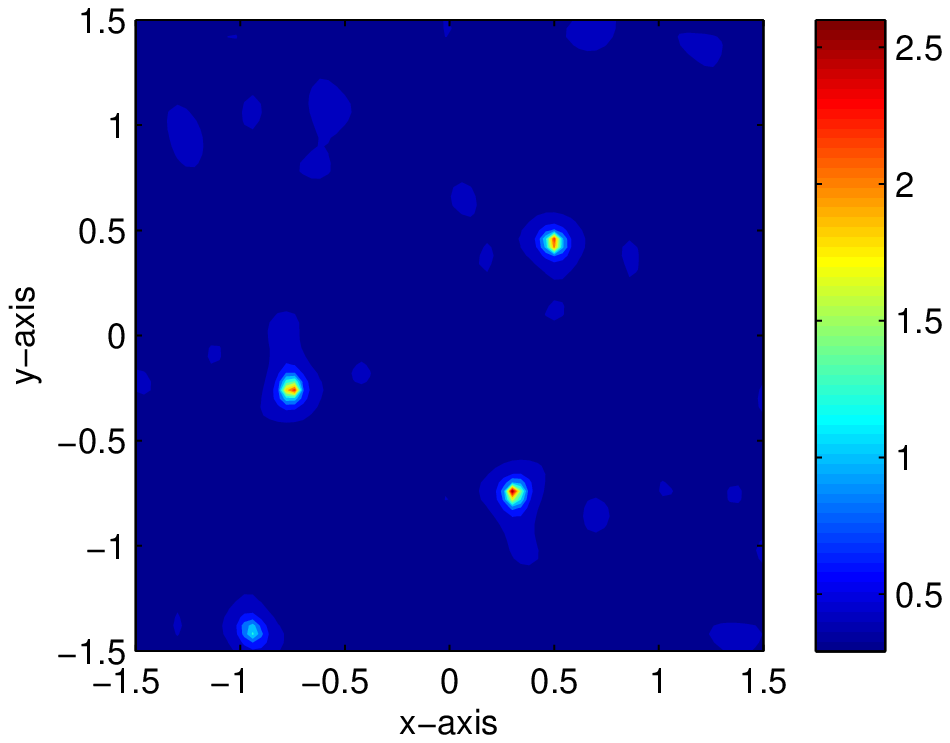}}
\subfloat[$\eta=15$]{\label{Result1-2}\includegraphics[width=0.495\textwidth]{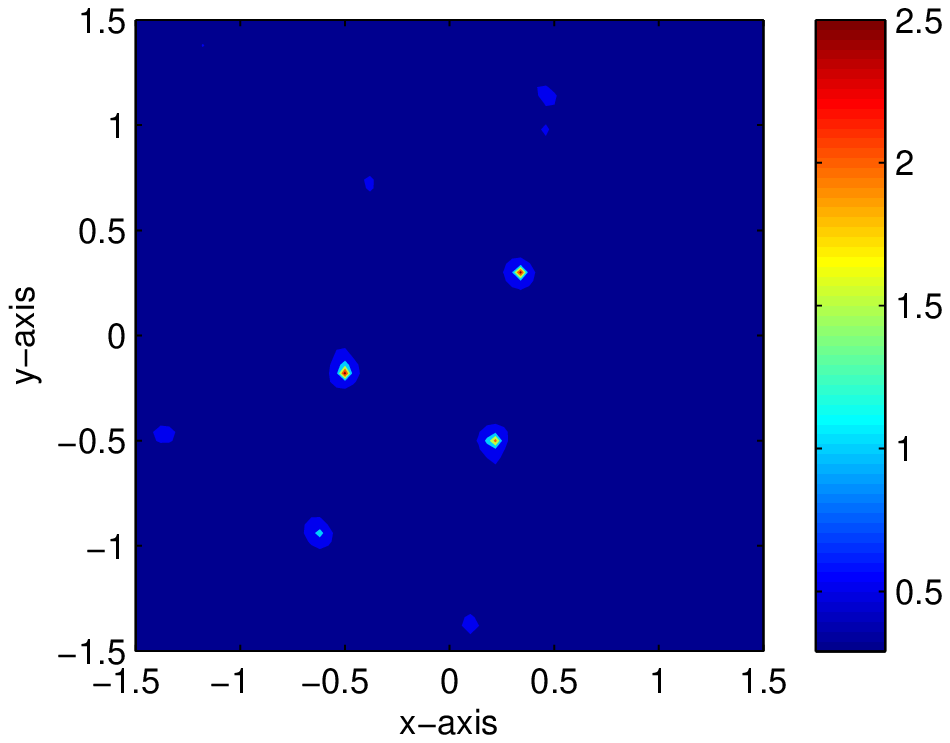}}\\
\subfloat[$\eta=20$]{\label{Result1-3}\includegraphics[width=0.495\textwidth]{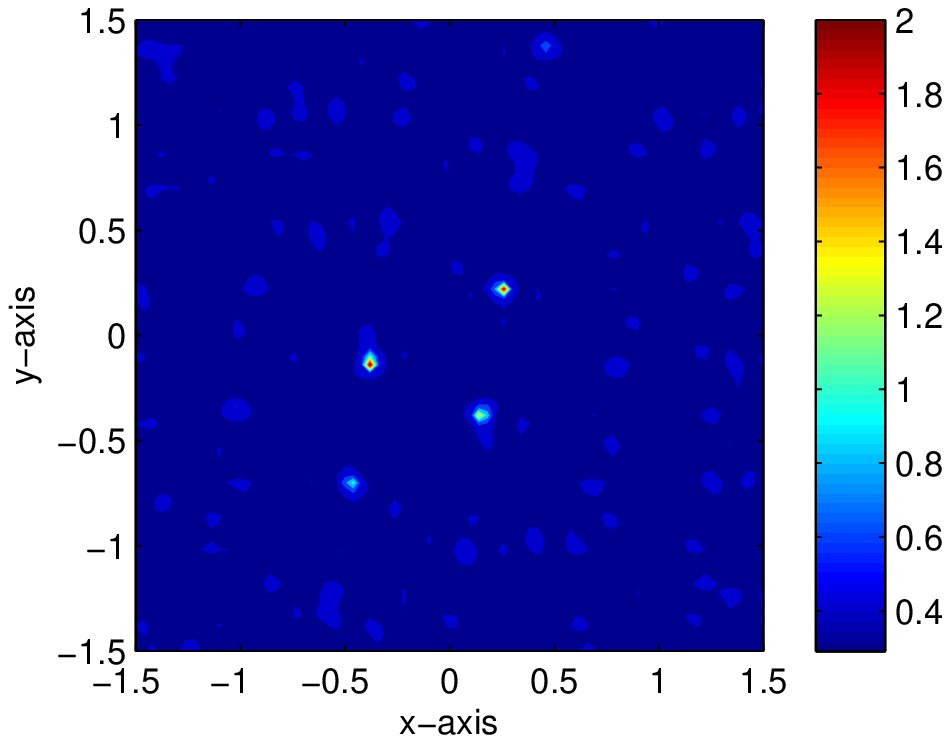}}
\subfloat[$\eta=k$]{\label{Result1-4}\includegraphics[width=0.495\textwidth]{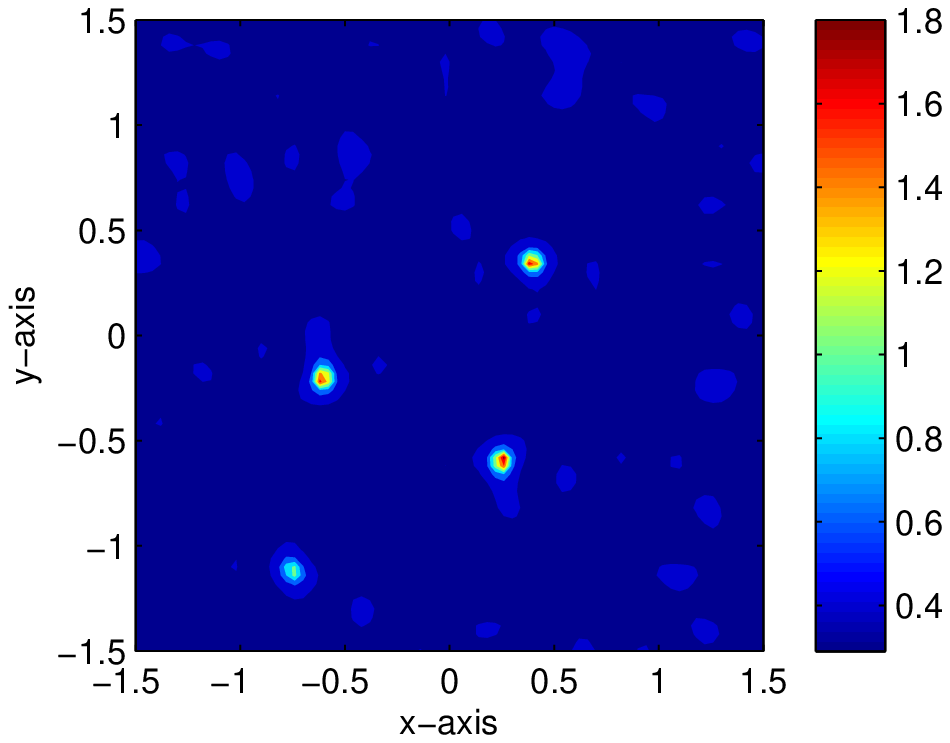}}
\caption{Maps of $\mathbb{E}(\mz;\eta)$ for various value of $\eta$ when true value of $k$ is $2\pi/0.5$.}\label{Result1}
\end{center}
\end{figure}

\begin{figure}
\begin{center}
\subfloat[$\eta=10$]{\label{Result2-1}\includegraphics[width=0.495\textwidth]{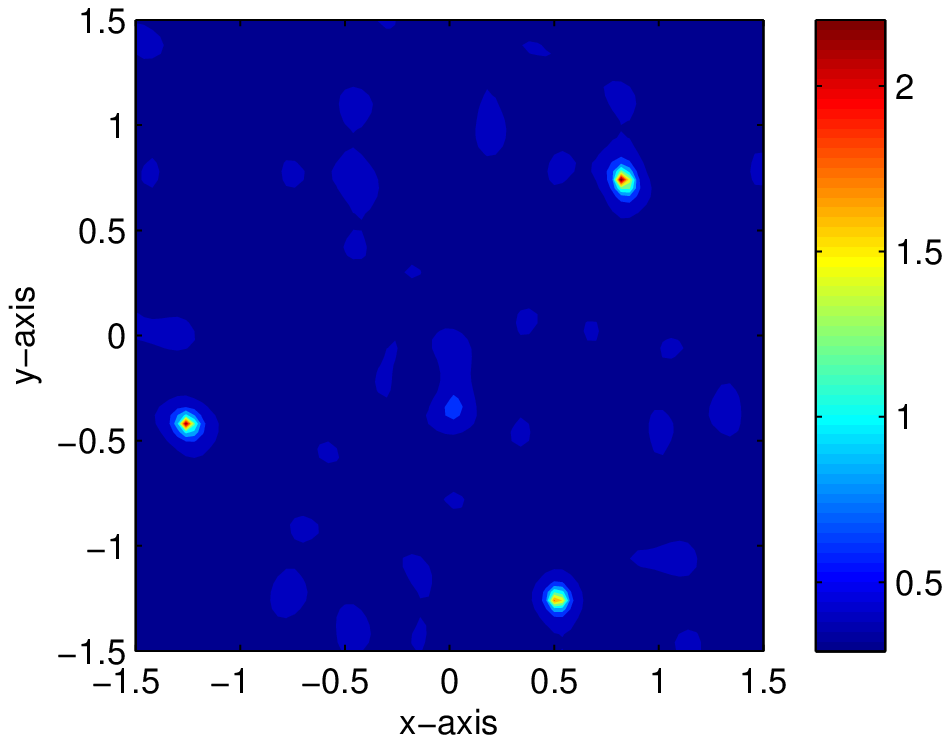}}
\subfloat[$\eta=15$]{\label{Result2-2}\includegraphics[width=0.495\textwidth]{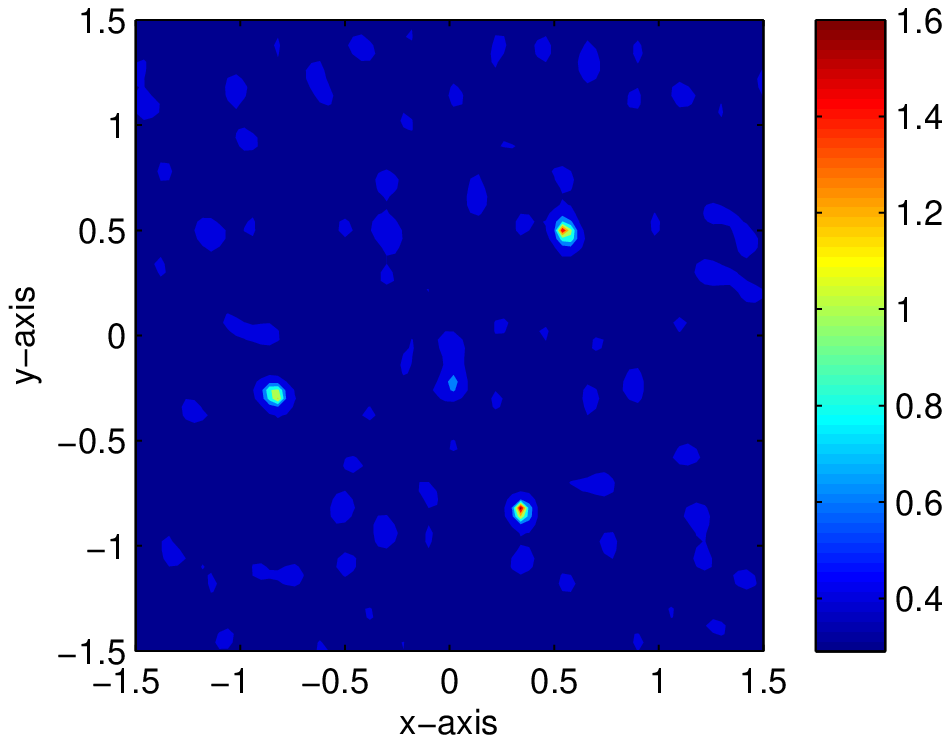}}\\
\subfloat[$\eta=20$]{\label{Result2-3}\includegraphics[width=0.495\textwidth]{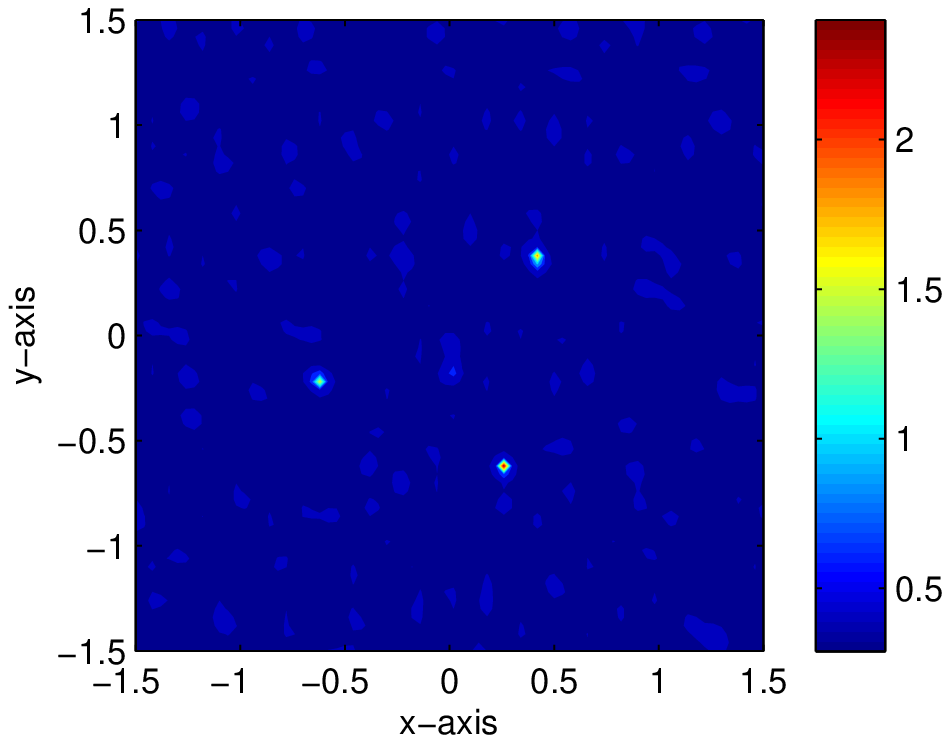}}
\subfloat[$\eta=k$]{\label{Result2-4}\includegraphics[width=0.495\textwidth]{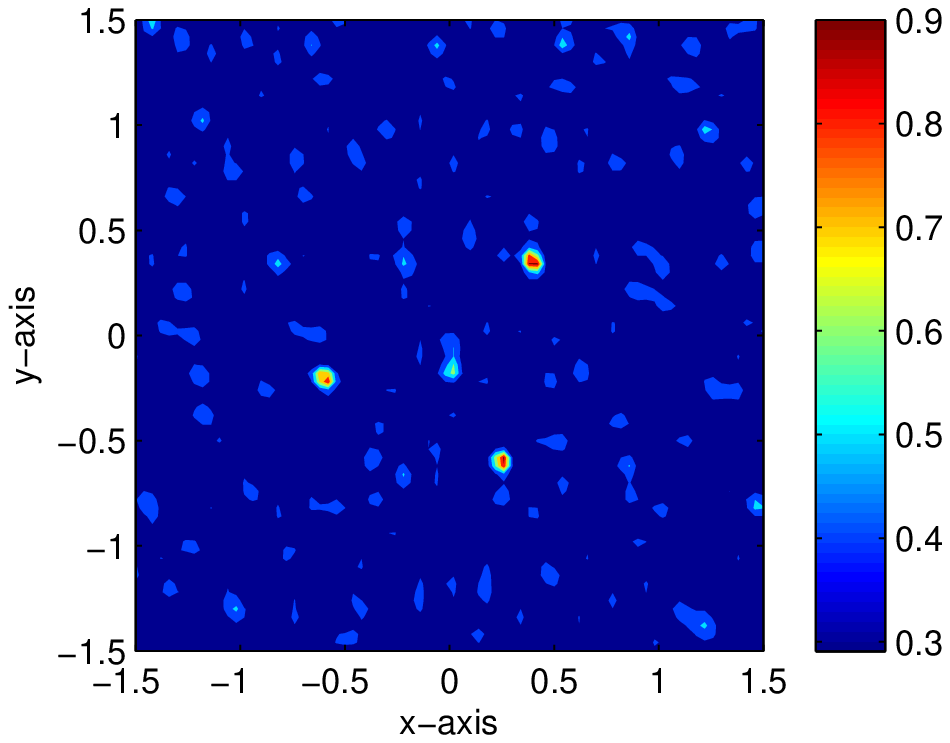}}
\caption{\label{Result2}
Same as Figure \ref{Result1} except true value of $k$ is $2\pi/0.3$.}\end{center}
\end{figure}

Following \cite{PL1,PL3}, MUSIC can be applied to identify the shapes of extended cracks or thin electromagnetic inhomogeneities. Thus, we consider shape identification of the arc-like, extended crack
\[\Gamma_{\mbox{\small ext}}=\set{[s,0.5\cos(0.5s\pi)+0.2\sin(0.5s\pi)-0.1\cos(1.5s\pi)]^T~:~-1\leq s\leq1}.\]
Figure \ref{Result3} shows the distribution of normalized singular values of $\mathbb{K}$ and maps of $\mathbb{E}(\mz;\eta)$ for various values of $\eta$, $N=32$ directions, and $k=2\pi/0.4\approx15.7080$. Based on \cite{PL1}, the first $M=13$ left-singular vectors were selected to define the projection operator in (\ref{Projection}).

Similar to the previous examples, since the value of $\eta$ is different from $k$, we cannot identify the true shape of $\Gamma_{\mbox{\small ext}}$. Note that if $\eta=10$, since $\eta<k$, $\mathbb{E}(\mx;\eta)$ plots large values at $(k/\eta)\mz$, $\mz\in\Gamma_{\mbox{\small ext}}$, i.e., the identified shape of $\Gamma_{\mbox{\small ext}}$ must be larger than the true shape of $\Gamma_{\mbox{\small ext}}$. In contrast, if $\eta=20$ or $25$, since $\eta>k$, the identified shape of $\Gamma_{\mbox{\small ext}}$ must be smaller than the true shape of $\Gamma_{\mbox{\small ext}}$. Since $\eta\approx k$ when $\eta=15$, the identified shape is approximately identical to the true shape of $\Gamma_{\mbox{\small ext}}$; however, the exact shape of $\Gamma_{\mbox{\small ext}}$ is still unknown. Hence, for any value of $\eta\ne k$, the shape of existing crack $\Gamma_{\mbox{\small ext}}$ can be outlined, but identifying its true shape is impossible without exact value of $k$.

\begin{figure}
\begin{center}
\subfloat[distribution of normalized singular values]{\includegraphics[width=0.495\textwidth]{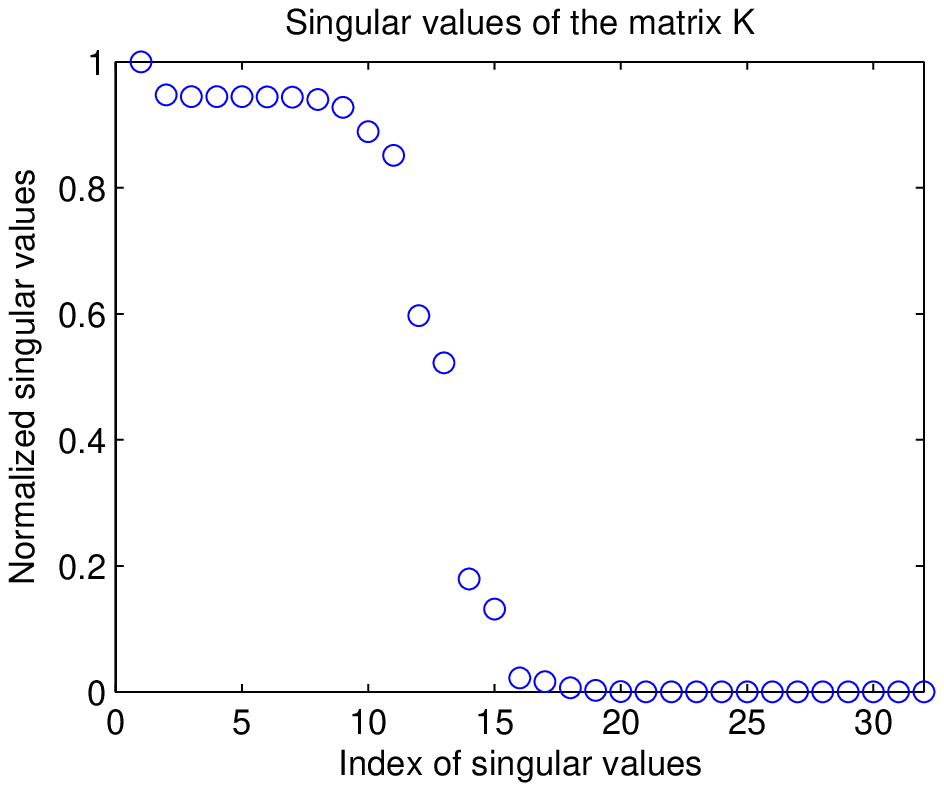}}
\subfloat[$\eta=10$]{\includegraphics[width=0.495\textwidth]{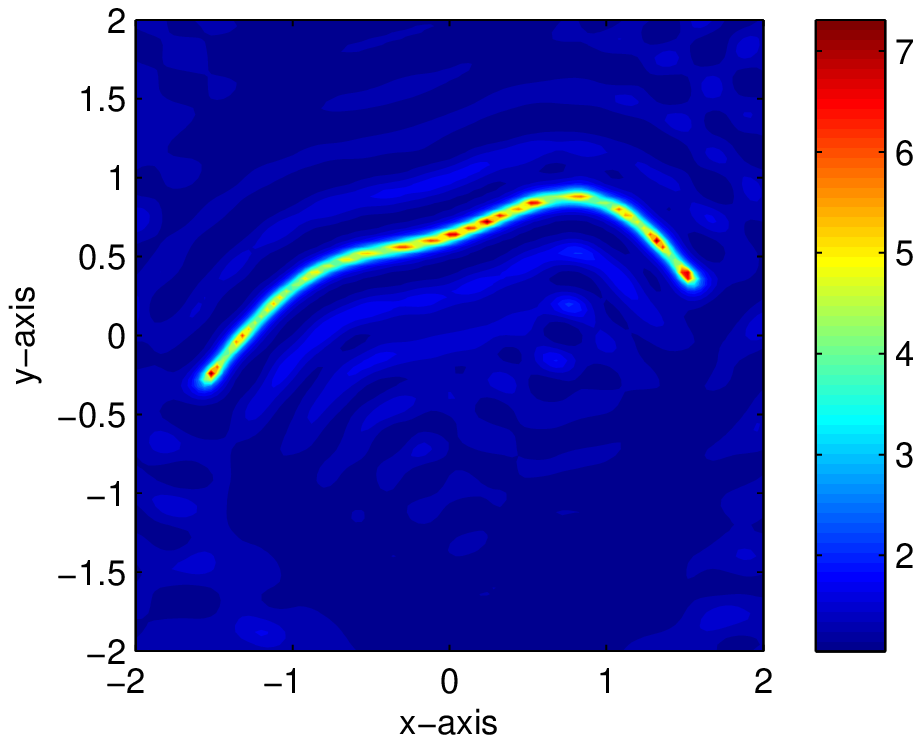}}\\
\subfloat[$\eta=15$]{\includegraphics[width=0.495\textwidth]{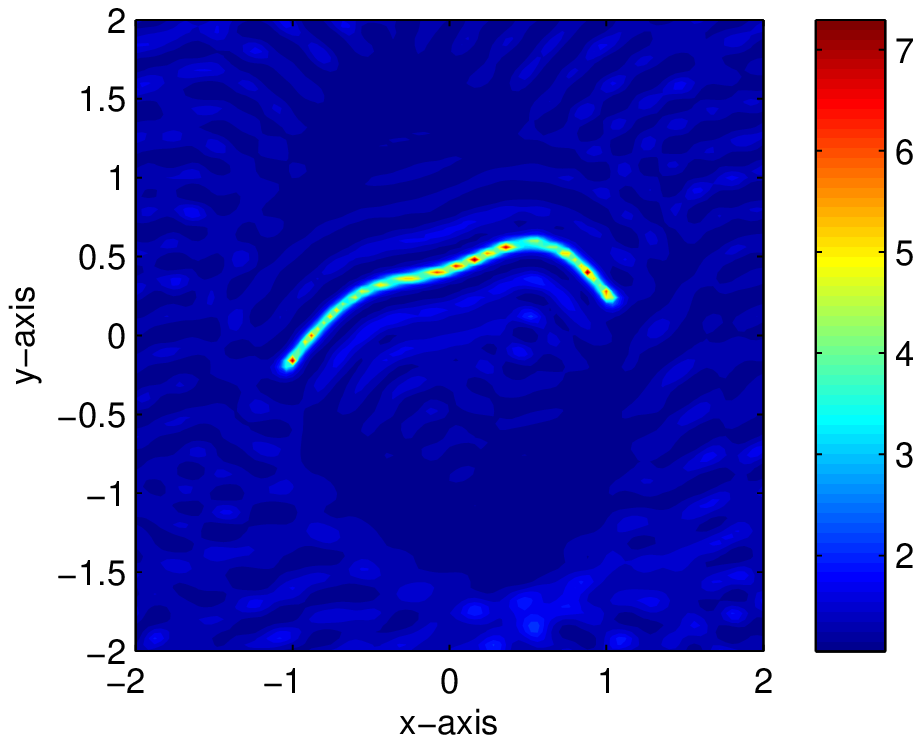}}
\subfloat[$\eta=20$]{\includegraphics[width=0.495\textwidth]{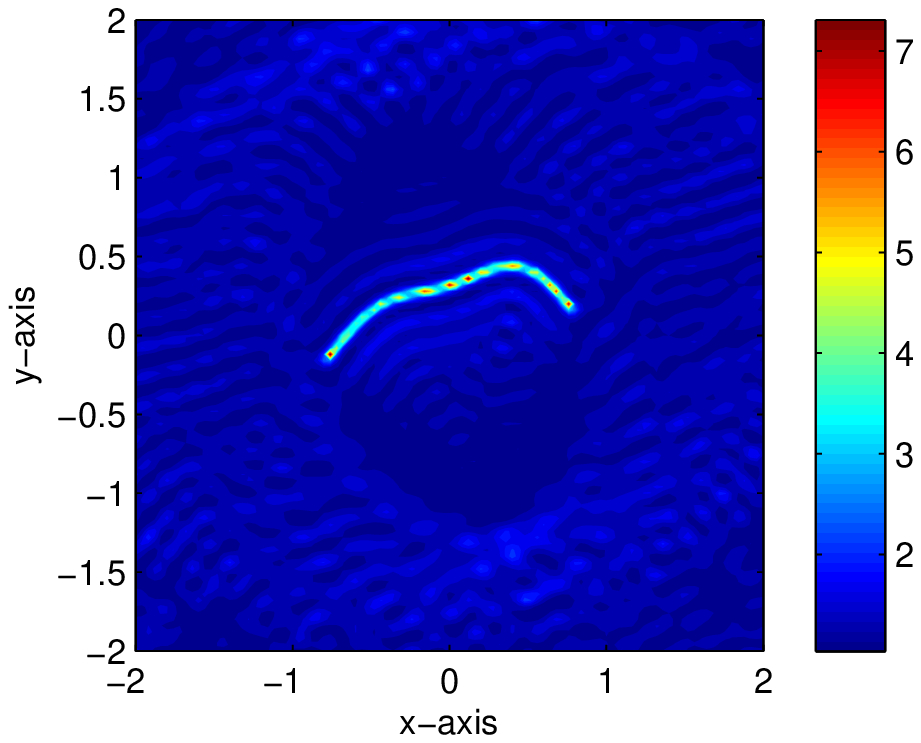}}\\
\subfloat[$\eta=25$]{\includegraphics[width=0.495\textwidth]{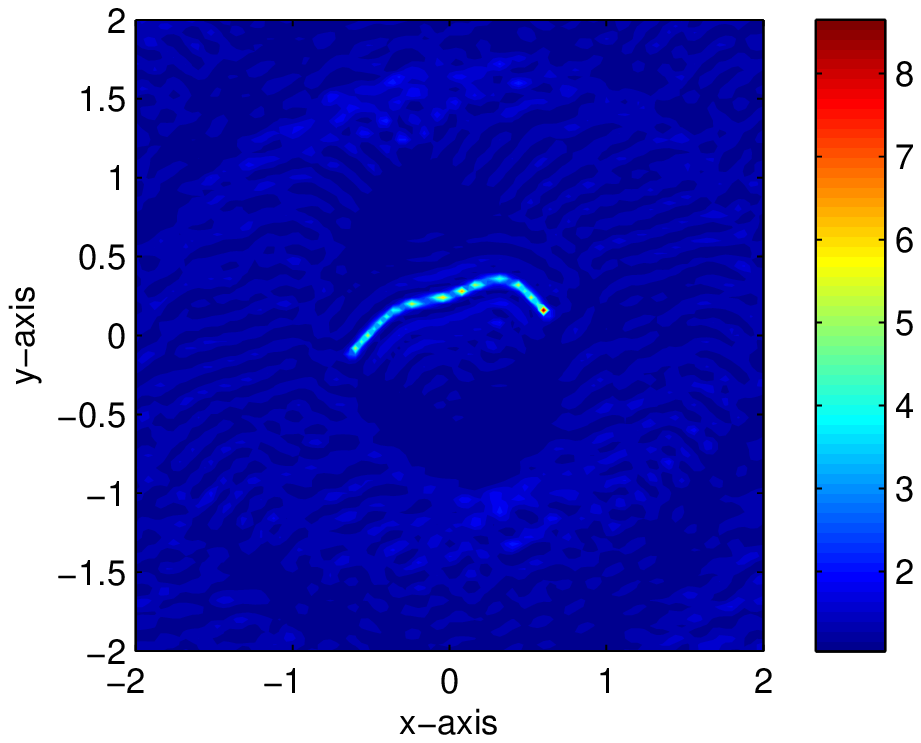}}
\subfloat[$\eta=k$]{\includegraphics[width=0.495\textwidth]{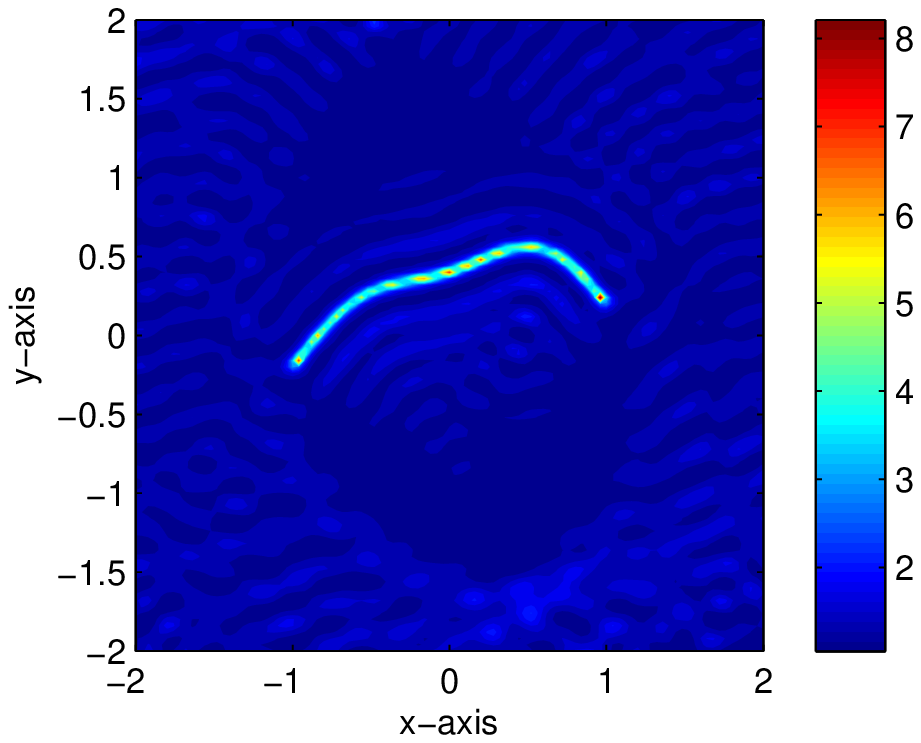}}
\caption{\label{Result3}Distribution of normalized singular values and maps of $\mathbb{E}(\mz;\eta)$ when true value of $k$ is $2\pi/0.4$.}
\end{center}
\end{figure}

On the basis of previous results, we can conclude that if one cannot approximate true value $k$, it is impossible to determine true location/shape of cracks. Hence, let us consider a simple algorithm to estimate the value of $k$ by creating a small scatterer $\Sigma$ located at $\my$. Since, for any positive value $\eta$, we can localize $(k/\eta)\my$ via the map of $\mathbb{E}(\mz;\eta)$ so that we can estimate $k\my$. By using \textit{a priori} information of location $\my$, we can estimate $k$ so that it will correspondingly be possible to identify true shape of $\Gamma_{\mbox{\small ext}}$.

Notice that for creating small scatterer $\Sigma$, we must select location $\my$ such that $\my\notin\Gamma_{\mbox{\small ext}}$. If not, we cannot apply this method. Unfortunately, we have no \textit{a priori} information of $\Gamma_{\mbox{\small ext}}$, creating $\Sigma$ must be considered carefully. Based on the previous result, we have information of $(k/\eta)\mz\in\Gamma_{\mbox{\small ext}}$, i.e., for any value of $\eta$, we can find two end-points, say $(k/\eta)\mz_1$ and $(k/\eta)\mz_2$, where $\mz_1$ and $\mz_2$ are end-points of $\Gamma_{\mbox{\small ext}}$. With this, we can generate two lines $L_1$ and $L_2$ such that
\[(k/\eta)\mz_1\in L_1\quad\mbox{and}\quad (k/\eta)\mz_2\in L_2\quad\mbox{for any }\eta\geq0.\]
This two lines separate $\mathbb{R}^2$ into two non-overlapped domains $\Omega$ and $\mathbb{R}^2\backslash\overline{\Omega}$ such that $(k/\eta)\mz\in\Omega$ for any $\eta\geq0$ and $\mz\in\Gamma$, refer to Figure \ref{Algorithm}. Thus, by creating $\Sigma$ into the $\mathbb{R}^2\backslash\overline{\Omega}$, we can estimate $k$ and correspondingly almost accurate shape of $\Gamma_{\mbox{\small ext}}$.

\begin{figure}
\begin{center}
\subfloat[Identified crack]{\includegraphics[width=0.495\textwidth]{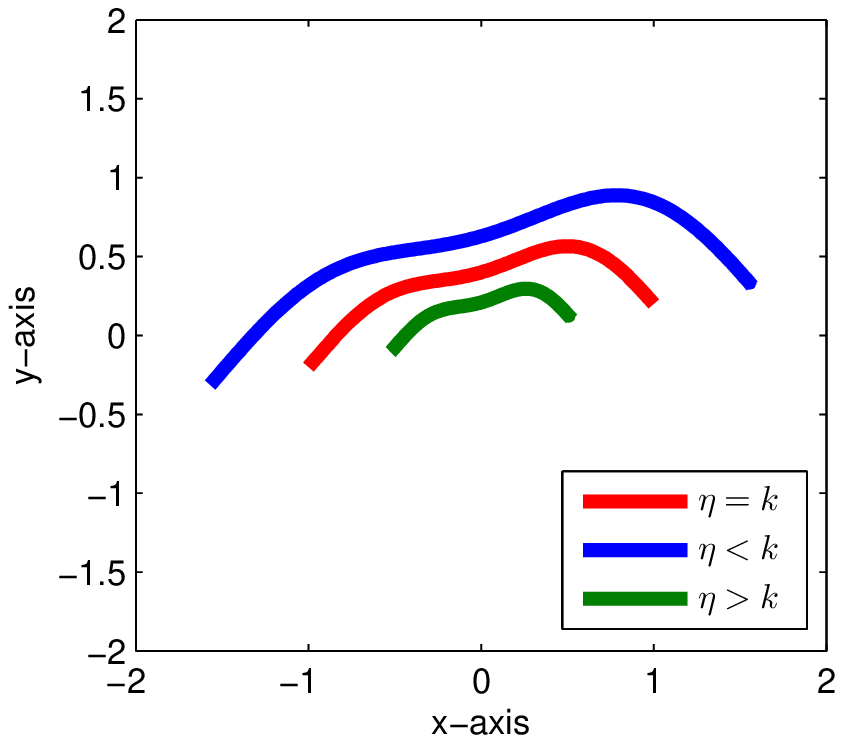}}
\subfloat[Creating small scatterer]{\includegraphics[width=0.495\textwidth]{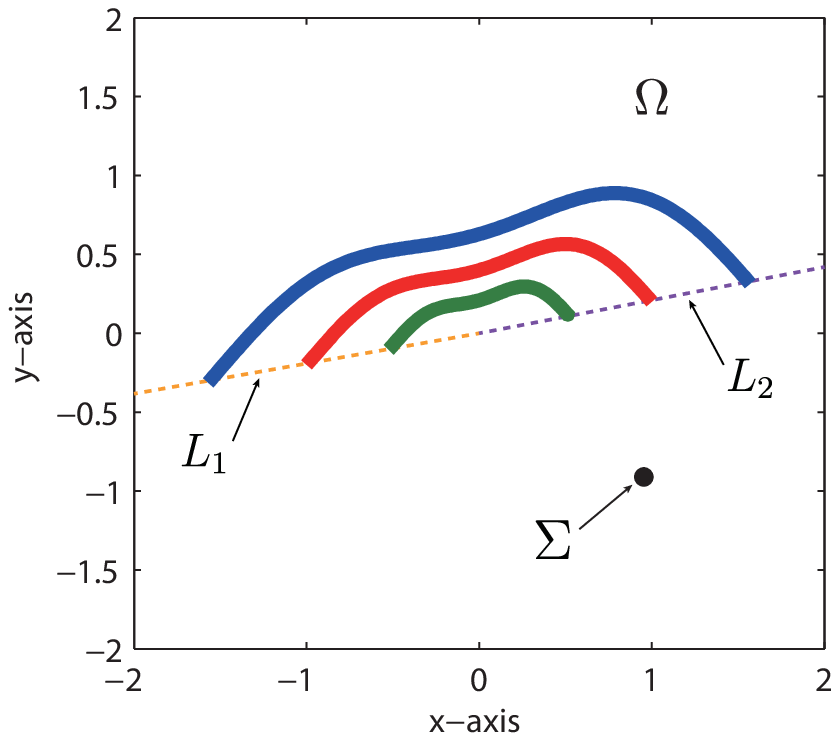}}
\caption{\label{Algorithm}Identified crack and creating small scatterer $\Sigma$.}
\end{center}
\end{figure}

At this stage, we approximate true value of $k$ for a proper imaging of $\Gamma_{\mbox{\small ext}}$. Opposite to the discussed algorithm, due to the unexpected artifacts in the map of $\mathbb{E}(\mz;\eta)$, it is very hard to identify small scatterer $\Sigma$. So, in this example, let us create an extended crack
\[\Sigma:=\set{[s,-1]^T~:~-1\leq s\leq1}.\]
Notice that $\Sigma$ is centered at $[0,-1]^T$ but its location is shifted to $[0,-0.77]^T$ in the map of $\mathbb{E}(\mz;20)$. This means that we can estimate $k$ such that
\[\left(\frac{k}{\eta}\right)\my=[0,-0.77]^T\quad\mbox{implies}\quad k=15.4000\approx15.7080=\frac{2\pi}{0.4}\]
and correspondingly, obtained shape of crack in the map of $\mathbb{E}(\mz;15.4000)$ is very accurate to the $\Gamma_{\mbox{\small ext}}$, refer to Figure \ref{Result4}.

\begin{figure}
\begin{center}
\subfloat[distribution of normalized singular values]{\includegraphics[width=0.495\textwidth]{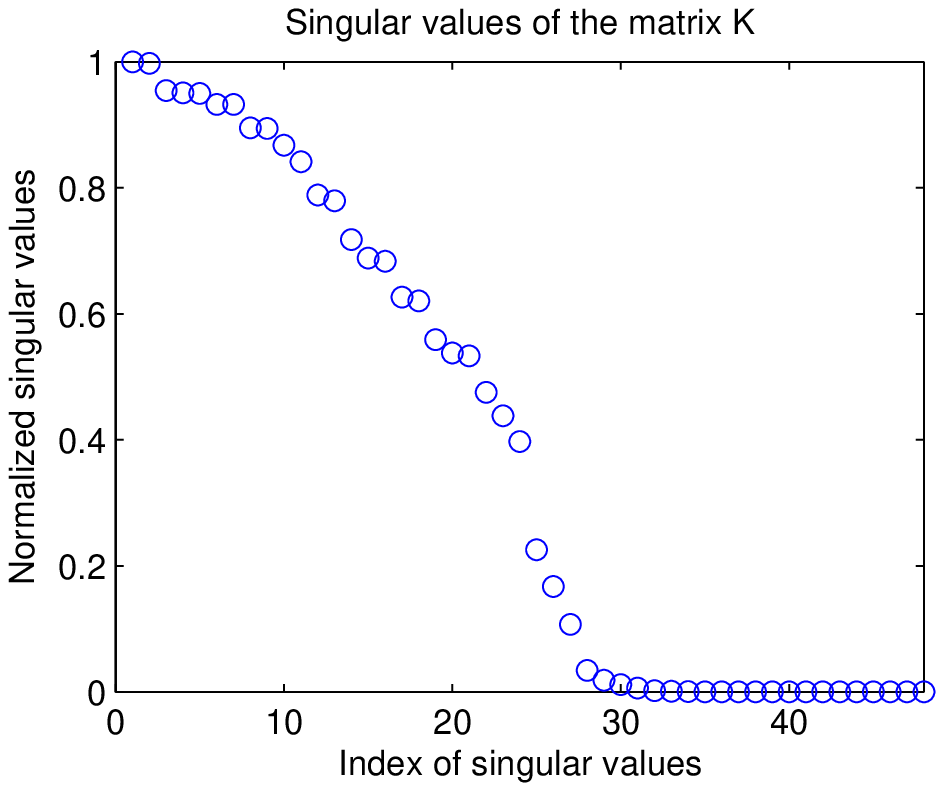}}
\subfloat[$\eta=20$]{\includegraphics[width=0.495\textwidth]{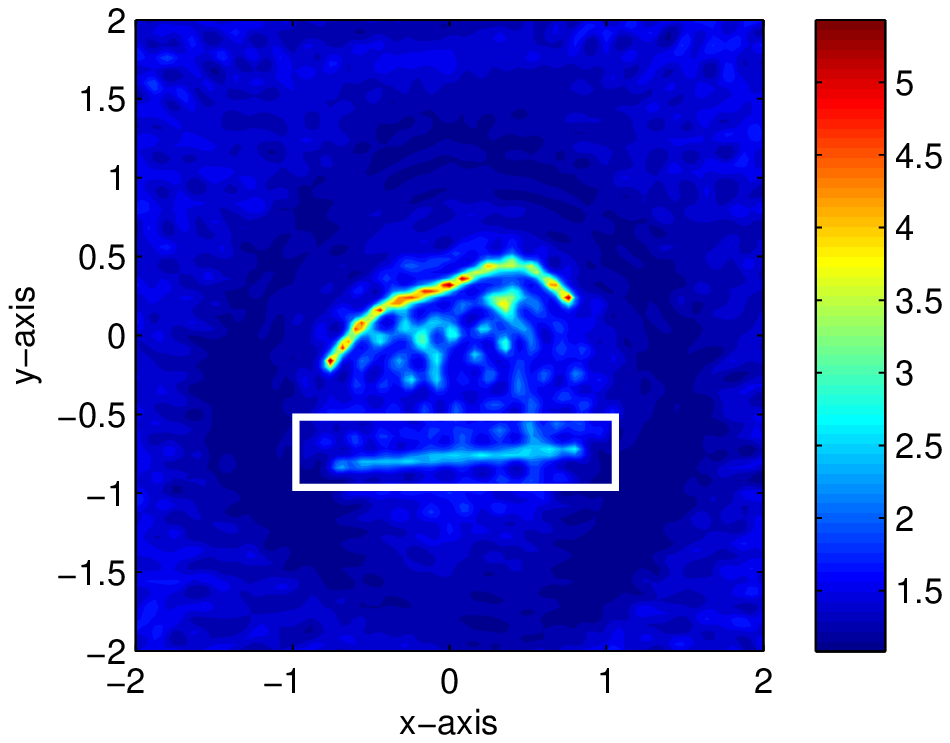}}\\
\subfloat[location $(k/20)\my$]{\label{Result4-3}\includegraphics[width=0.495\textwidth]{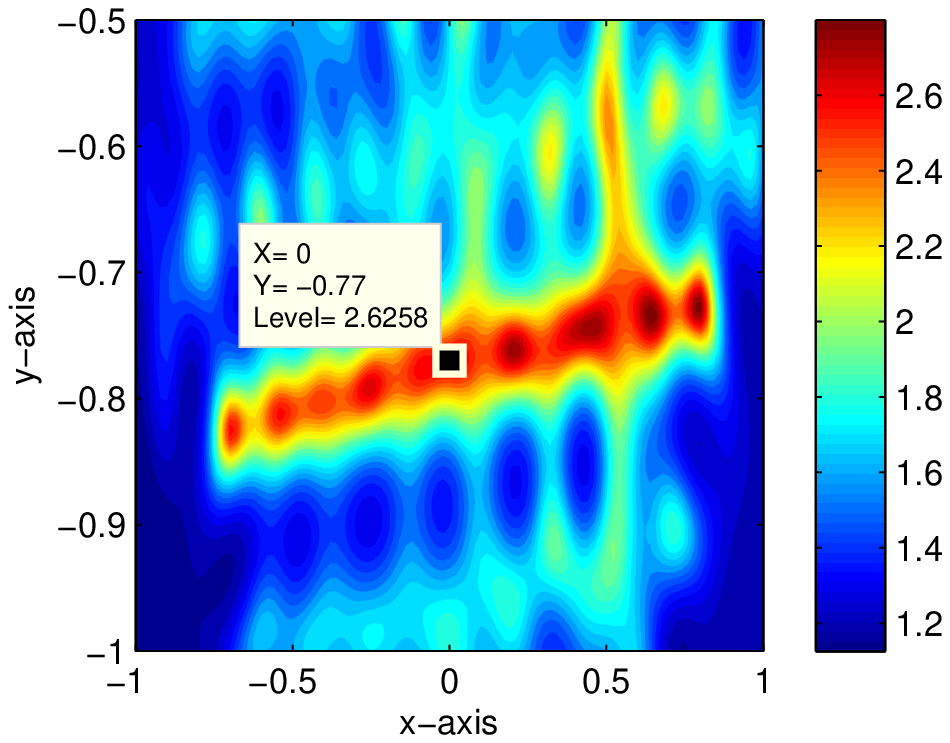}}
\subfloat[$\eta=15.4000$]{\includegraphics[width=0.495\textwidth]{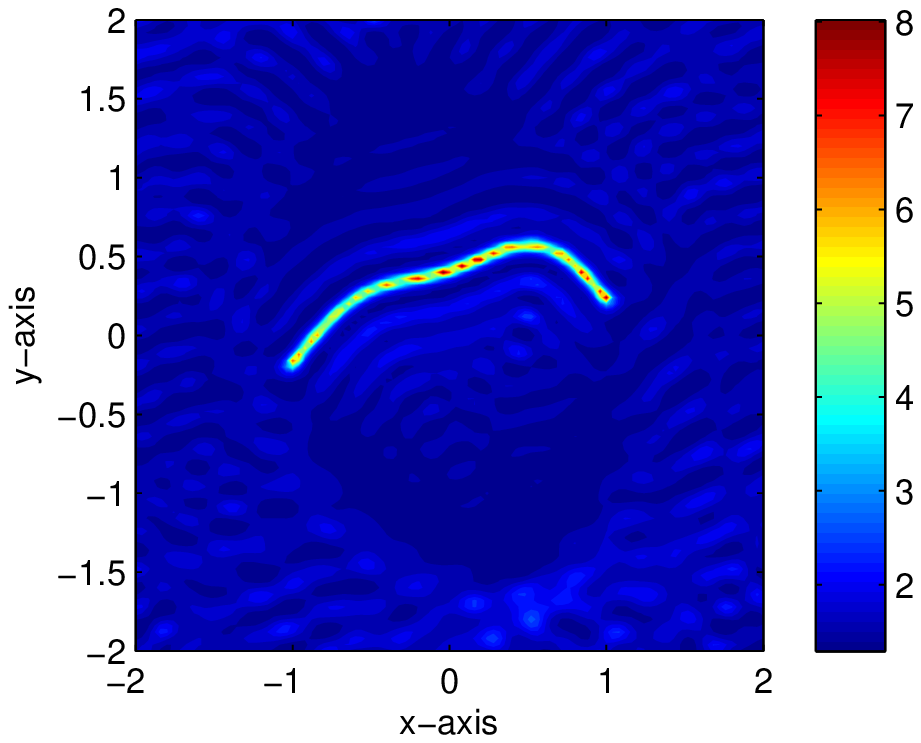}}
\caption{\label{Result4}Approximating true shape of $\Gamma_{\mbox{\small ext}}$.}
\end{center}
\end{figure}

\section{Conclusion}\label{sec5}
Based on the asymptotic expansion formula, the structure of a MUSIC-type imaging functional with an unknown wavenumber is derived by establishing a relationship with a Bessel function of order zero. Based on the identified structure, we determined why inaccurate locations/shapes of cracks were appeared and developed a simple algorithm for finding the exact locations/shapes of cracks by generating small or extended crack on the basis of the tendency of the inaccurate result.

Our current research considers cracks with Dirichlet boundary conditions; we plan to extend our research to cracks with Neumann boundary conditions (sound-hard arc in inverse acoustic scattering problem), refer to \cite{M1,M2}. Furthermore, motivated by \cite{AIL1,CI,KP,PP2,SCC}, extending our research to the limited-view inverse scattering problem will be an interesting research topic.

\bibliographystyle{degruyter-plain}
\bibliography{../../../References}

\providecommand{\bysame}{\leavevmode\hbox to3em{\hrulefill}\thinspace}
\providecommand{\MR}{\relax\ifhmode\unskip\space\fi MR }
\providecommand{\MRhref}[2]{%
  \href{http://www.ams.org/mathscinet-getitem?mr=#1}{#2}
}
\providecommand{\href}[2]{#2}
\begin{thebibliography}{10}

\bibitem{AGKPS}
H.~Ammari, J.~Garnier, H.~Kang, W.-K. Park and K.~S{\o}lna, Imaging schemes for
  perfectly conducting cracks, \emph{SIAM J. Appl. Math.} \textbf{71} (2011),
  68--91.

\bibitem{AIL1}
H.~Ammari, E.~Iakovleva and D.~Lesselier, A {MUSIC} algorithm for locating
  small inclusions buried in a half-space from the scattering amplitude at a
  fixed frequency, \emph{Multiscale Model. Simul.} \textbf{3} (2005), 597--628.

\bibitem{AKLP}
H.~Ammari, H.~Kang, H.~Lee and W.-K. Park, Asymptotic imaging of perfectly
  conducting cracks, \emph{SIAM J. Sci. Comput.} \textbf{32} (2010), 894--922.

\bibitem{C1}
X.~Chen, Multiple signal classification method for detecting point-like
  scatterers embedded in an inhomogeneous background medium, \emph{J. Acoust.
  Soc. Am.} \textbf{127} (2010), 2392--2397.

\bibitem{CZ}
X.~Chen and Y.~Zhong, {MUSIC} electromagnetic imaging with enhanced resolution
  for small inclusions, \emph{Inverse Problems} \textbf{25} (2009), 015008.

\bibitem{C}
M.~Cheney, The linear sampling method and the {MUSIC} algorithm, \emph{Inverse
  Problems} \textbf{17} (2001), 591--595.

\bibitem{CI}
M.~Cheney and D.~Issacson, Inverse problems for a perturbed dissipative
  half-space, \emph{Inverse Problems} \textbf{11} (1995), 865--888.

\bibitem{K1}
A.~Kirsch, The {MUSIC} algorithm and the factorization method in inverse
  scattering theory for inhomogeneous media, \emph{Inverse Problems}
  \textbf{18} (2002), 1025--1040.

\bibitem{K}
R.~Kress, Inverse scattering from an open arc, \emph{Math. Meth. Appl. Sci.}
  \textbf{18} (1995), 267--293.

\bibitem{KP}
Y.~M. Kwon and W.-K. Park, Analysis of subspace migration in limited-view
  inverse scattering problems, \emph{Appl. Math. Lett.} \textbf{26} (2013),
  1107--1113.

\bibitem{M2}
L.~M{\"o}nch, On the numerical solution of the direct scattering problem for an
  open sound-hard arc, \emph{J. Comput. Appl. Math.} \textbf{17} (1996),
  343--356.

\bibitem{M1}
L.~M{\"o}nch, On the inverse acoustic scattering problem by an open arc: the
  sound-hard case, \emph{Inverse Problems} \textbf{13} (1997), 1379--1392.

\bibitem{N}
Z.~T. Nazarchuk, \emph{Singular Integral Equations in Diffraction Theory},
  Mathematics and Applications Series, Karpenko Physicomechanical Institute,
  Ukrainian Academy of Sciences, Lviv, 1994.

\bibitem{P-MUSIC1}
W.-K. Park, Asymptotic properties of {MUSIC}-type imaging in two-dimensional
  inverse scattering from thin electromagnetic inclusions, \emph{SIAM J. Appl.
  Math.} \textbf{75} (2015), 209--228.

\bibitem{P-SUB3}
W.-K. Park, Multi-frequency subspace migration for imaging of perfectly
  conducting, arc-like cracks in full- and limited-view inverse scattering
  problems, \emph{J. Comput. Phys.} \textbf{283} (2015), 52--80.

\bibitem{PL1}
W.-K. Park and D.~Lesselier, Electromagnetic {MUSIC}-type imaging of perfectly
  conducting, arc-like cracks at single frequency, \emph{J. Comput. Phys.}
  \textbf{228} (2009), 8093--8111.

\bibitem{PL3}
W.-K. Park and D.~Lesselier, {MUSIC}-type imaging of a thin penetrable
  inclusion from its far-field multi-static response matrix, \emph{Inverse
  Problems} \textbf{25} (2009), 075002.

\bibitem{PP2}
W.-K. Park and T.~Park, Multi-frequency based direct location search of small
  electromagnetic inhomogeneities embedded in two-layered medium, \emph{Comput.
  Phys. Commun.} \textbf{184} (2013), 1649--1659.

\bibitem{SRACM}
R.~Solimene, G.~Ruvio, A.~Dell'Aversano, A.~Cuccaro, M.~J. Ammann and
  R.~Pierri, Detecting point-like sources of unknown frequency spectra,
  \emph{Prog. Electromagn. Res. B} \textbf{50} (2013), 347--364.

\bibitem{SCC}
R.~Song, R.~Chen and X.~Chen, Imaging three-dimensional anisotropic scatterers
  in multi-layered medium by {MUSIC} method with enhanced resolution, \emph{J.
  Opt. Soc. Am. A} \textbf{29} (2012), 1900--1905.

\bibitem{ZC}
Y.~Zhong and X.~Chen, {MUSIC} imaging and electromagnetic inverse scattering of
  multiple-scattering small anisotropic spheres, \emph{IEEE Trans. Antennas
  Propag.} \textbf{55} (2007), 3542--3549.

\end{thebibliography}
%
%
%
%
%
%
%

\end{document}